\documentclass[a4paper,10pt]{article}

\usepackage{graphicx}
\usepackage{geometry}
\usepackage{amsthm}
\usepackage{amssymb}
\usepackage{amsmath}
\usepackage{hyperref}
\usepackage{xcolor}
\usepackage{enumerate}
\usepackage{listings}
\usepackage{complexity}
\usepackage{blkarray}
\usepackage{lmodern}

\usepackage[font=small,labelfont=bf]{caption}
\usepackage[font=small,labelfont=normalfont,labelformat=simple]{subcaption}
\graphicspath{{figs/}}

\newtheorem{theorem}{Theorem}

\newtheorem{lemma}[theorem]{Lemma}

\newtheorem{conjecture}{Conjecture}

\author
{
Raphael Steiner \thanks{Institute of Mathematics, Technische Universit\"at Berlin, Germany, email: \texttt{steiner@math.tu-berlin.de}.
Funded by DFG-GRK 2434 Facets of Complexity.}
}

\date{\today}

\title{A Note on Coloring Digraphs of Large Girth}

\begin{document}
\maketitle

\begin{abstract}
The \emph{digirth} of a digraph is the length of a shortest directed cycle. The \emph{dichromatic number} $\vec{\chi}(D)$ of a digraph $D$ is the smallest size of a partition of the vertex-set into subsets inducing acyclic subgraphs.
A conjecture by Harutyunyan and Mohar~\cite{HarutyunyanMohar} states that $\vec{\chi}(D) \le \lceil\frac{\Delta}{4}\rceil+1$ for every digraph $D$ of digirth at least $3$ and maximum degree $\Delta$. The best known partial result by Golowich~\cite{Golowich} shows that $\vec{\chi}(D) \le \frac{2}{5}\Delta+O(1)$. 
In this short note we prove for every $g \ge 2$ that if $D$ is a digraph of digirth at least $2g-1$ and maximum degree $\Delta$, then $\vec{\chi}(D) \le (\frac{1}{3}+\frac{1}{3g}) \Delta + O_g(1)$. This improves the bound of Golowich for digraphs without directed cycles of length at most $10$.
\end{abstract}

\section{Introduction}
\paragraph{Preliminaries.}
All digraphs in this note are finite and do not contain loops or parallel arcs. Given a digraph $D$, we denote by $V(D)$ its vertex-set and by $A(D)$ the arc-set. A digraph is called \emph{acyclic} if it does not contain directed cycles. By $\Delta(D), \Delta^+(D), \Delta^-(D), \delta^+(D), \delta^-(D)$ we denote, respectively, the maximum degree in (the underlying graph of) $D$, and the extremal out- and in-degrees in $D$. We furthermore denote by $\tilde{\Delta}(D)=\max\{\sqrt{d^+(v)d^-(v)}|v \in V(D)\}$ the maximum geometric mean of the in- and out-degree of a vertex in $D$. Note that in case $D$ has no cycles of length $2$, the inequality of geometric and arithmetic mean shows that  $\tilde{\Delta}(D)\le \frac{\Delta(D)}{2}$. Given a vertex set $X \subseteq V(D)$, we denote by $D[X]$ the induced subdigraph of $D$ with vertex-set $X$ and call $X$ \emph{acyclic} if $D[X]$ is acyclic. By $\vec{g}(D)$ we denote the \emph{digirth} of $D$, that is, the shortest length of a directed cycle in $D$ ($\vec{g}(D):=\infty$ if $D$ is acyclic). Given a a family $A_1,\ldots,A_m$ of finite sets, a \emph{system of representatives} of this family is a set $X \subseteq \bigcup_{i=1}^{m}{A_i}$ such that $X \cap A_i \neq \emptyset$ for all $i \in [m]$.
\newline\newline
We deal with a notion of coloring for directed graphs introduced in 1982 by Neumann-Lara~\cite{neulara}. Given a digraph $D$, an \emph{acyclic} coloring of $D$ is a vertex-coloring in which all color classes are acyclic. The smallest number of colors sufficient for an acyclic coloring of $D$ is denoted by $\vec{\chi}(D)$ and called \emph{dichromatic number} of $D$. This notion has received a fair amount of attention in the past two decades, see \cite{largesubdivisions, perfect, lists, HARUTYUNYAN2019, dig5, dig4, fractionalNL} for some recent results. As for undirected graphs, there is a Brooks-type upper bound on the dichromatic number of a digraph, see~\cite{MoharBrooks, neulara}, which implies $\vec{\chi}(D) \le \lceil\frac{\Delta}{2}\rceil$ for every digraph of girth at least $3$ and maximum degree $\Delta \ge 3$.
In this note, we are motivated by the following conjecture from \cite{HarutyunyanMohar}, which claims that this Brook's type bound can be improved by a factor of $2$ if we forbid directed cycles of length $2$ in the digraph.
\begin{conjecture}[cf. \cite{HarutyunyanMohar}, Conjecture~1.5]
Let $D$ be a digraph of digirth at least $3$ and maximum degree $\Delta$. Then $\vec{\chi}(D) \le \left\lceil \frac{\tilde{\Delta}(D)}{2}\right\rceil +1 \le \lceil \frac{\Delta}{4}\rceil +1$.
\end{conjecture}
Approaching their conjecture, in \cite{HarutyunyanMohar} Harutyunyan and Mohar proved that there is a small absolute constant $\varepsilon>0$ such that $\vec{\chi}(D) \le (1-\varepsilon)\tilde{\Delta}(D) \le (\frac{1}{2}-\frac{\varepsilon}{2})\Delta(D)$ for every digraph $D$ of digirth at least $3$ and $\tilde{\Delta}$ sufficiently large. Subsequently Golowich~\cite{Golowich} improved the multiplicative constant in the upper bound, by showing that every digraph $D$ of digirth at least $3$ satisfies $\vec{\chi}(D) \le \frac{2}{5}\Delta(D)+O(1)$. Our contribution is to further improve the multiplicative constant in this upper bound for digraphs without short directed cycles.
\begin{theorem}\label{main}
Let $g \ge 2$ a natural number, and let $D$ be a digraph with $\vec{g}(D) \ge 2g-1$ and maximum degree $\Delta$. Then
$\vec{\chi}(D) \le (\frac{1}{3}+\frac{1}{3g})\Delta+(g+1)$.
\end{theorem}
\section{Proof of Theorem~\ref{main}}
We need three auxiliary results by Neumann-Lara, by Aharoni, Berger and Kfir, and by Lov\'{a}sz.
\begin{lemma}[cf. \cite{neulara}, Theorem 5]\label{critical}
Let $k \in \mathbb{N}$ and let $D$ be a $(k+1)$-critical digraph, that is, $\vec{\chi}(D)=k+1$ but $\vec{\chi}(D') \le k$ for every proper subdigraph $D' \subsetneq D$. Then $\delta^+(D), \delta^-(D) \ge k$.
\end{lemma}
\begin{lemma}[cf. \cite{Aharoni}, Corollary II.13]\label{acyclichittingset}
Let $D$ be a digraph of digirth at least $\gamma \ge 2$ and let $V_1,V_2,\ldots,V_m$ be a partition of $V(D)$. If $|V_i| \ge \frac{\gamma}{\gamma-1}\Delta^+(D)$ for all $i \in [m]$, then there is a system $X$ of representatives of $V_1,\ldots,V_m$ which is acyclic in $D$.
\end{lemma}
\begin{lemma}[\cite{lovasz}]\label{maxdegreduction}
Let $G$ be an undirected graph, $k \in \mathbb{N}$. Then $V(G)$ admits a partition $X_1,\ldots,X_k$ such that for every $v \in X_i, i \in [k]$, we have $\deg_{G[X_i]}(v) \le \frac{1}{k}\deg(v)$.
\end{lemma}
The proof of Theorem~\ref{main} relies on the following bound on the dichromatic number for digraphs of large girth compared to their maximum out-degree.
\begin{lemma}\label{highgirth}
Let $D$ be a digraph such that $\vec{g}(D) > \Delta^+(D)$. Then 
$$\vec{\chi}(D) \le \left\lfloor \frac{\Delta(D)}{3}\right\rfloor +2.$$
\end{lemma}
\begin{proof}
Abbreviate $\Delta=\Delta(D)$ and $\gamma=\vec{g}(D)$ and put $k:=\left\lfloor \frac{\Delta}{3} \right\rfloor+1>\frac{\Delta}{3}$. By Lemma~\ref{maxdegreduction} there is a partition $X_1, \ldots, X_k$ of $V(D)$ such that for every $i \in [m]$ we have $\Delta(D[X_i]) \le \frac{\Delta}{k}<3$. Hence, $D[X_i]$ is a disjoint union of oriented paths and oriented cycles. For every $i$, let us denote by $\vec{\mathcal{C}}_i$ the set of all \emph{directed} cycles in $D[X_i]$ and put $V':=\bigcup_{i \in [k], C \in \vec{\mathcal{C}}_i}{V(C)}$. We claim that there is an acyclic set $X$ in $D$ such that $X \cap V(C) \neq \emptyset$ for all $C \in \vec{\mathcal{C}}_i$ and $i \in [k]$. To see this, note that $|V(C)| \ge \gamma \ge \frac{\gamma}{\gamma-1}\Delta^+(D) \ge \frac{\gamma}{\gamma-1}\Delta^+(D[V'])$ for every $C \in \vec{\mathcal{C}}_i$ and $i \in [k]$. We can therefore apply Lemma~\ref{acyclichittingset} to the digraph $D[V']$ equipped with the partition $(V(C)|C \in \vec{\mathcal{C}}_i, i \in [k])$ to find a system of representatives $X$ which is acyclic in $D[V']$ and thus in $D$. Next we claim that each of the sets $X_i \setminus X, i \in [k]$ is acyclic in $D$. Indeed, the digraph $D[X_i \setminus X]=D[X_i]-(X_i \cap X)$ is obtained from a disjoint union of oriented paths and cycles by removing at least one vertex from each directed cycle, and is therefore acyclic. Hence, $X_1\setminus X,X_2 \setminus X,\ldots,X_k \setminus X,X$ is a partition of $V(D)$ into acyclic sets which certifies that $\vec{\chi}(D) \le k+1=\left\lfloor \frac{\Delta}{3}\right\rfloor +2$.
\end{proof}
We can now complete the proof of Theorem~\ref{main} by applying Lemma~\ref{maxdegreduction} a second time.
\begin{proof}[Proof of Theorem~\ref{main}]
Let $\ell:=\left\lfloor\frac{\Delta}{3g}\right\rfloor+1$. By Lemma~\ref{maxdegreduction} there exists a partition $Y_1,\ldots,Y_\ell$ of $V(D)$ such that $\Delta(D[Y_i]) \le \frac{\Delta}{\ell}<3g$ for every $i \in [\ell]$. 
We claim that for every $i \in [\ell]$, we have $\vec{\chi}(D[Y_i]) \le g+1$. Suppose by way of a contradiction that $\vec{\chi}(D[Y_i]) \ge g+2$ for some $i \in [\ell]$. Consider a subgraph $D_i$ of $D[Y_i]$ with $\vec{\chi}(D_i) \ge g+2$ minimizing $|V(D_i)|+|A(D_i)|$. Clearly, $D_i$ is $(g+2)$-critical, and thus $\delta^-(D_i) \ge g+1$ by Lemma~\ref{critical}. Hence we have $$\Delta^+(D_i) \le \Delta(D_i)-\delta^-(D_i) \le \Delta(D[Y_i]) -\delta^-(D_i) \le (3g-1)-(g+1)=2g-2<\vec{g}(D) \le \vec{g}(D_i).$$ We can therefore apply Lemma~\ref{highgirth} to obtain $\vec{\chi}(D_i) \le \left\lfloor\frac{3g-1}{3}\right\rfloor+2=g+1$, which is the desired contradiction. This shows that indeed we have $\vec{\chi}(D[Y_i]) \le g+1$ for all $i \in [\ell]$. The claim now follows from $$\vec{\chi}(D) \le \sum_{i=1}^{\ell}{\vec{\chi}(D[Y_i])} \le (g+1)\left(\left\lfloor\frac{\Delta}{3g}\right\rfloor+1\right) \le \left(\frac{1}{3}+\frac{1}{3g}\right)\Delta+(g+1).$$
\end{proof}
\bibliographystyle{abbrv}
\bibliography{bibliography}
\end{document}